\documentclass[a4paper]{amsart}

\usepackage{amsmath}
\usepackage{amssymb}
\usepackage{amsthm}
\usepackage{bbm}

\newtheorem{theorem}{Theorem}

\newtheorem{lemma}[theorem]{Lemma}

\newcommand{\Oh}{\mathrm{O}}
\newcommand{\oh}{\mathrm{o}}
\newcommand{\im}{\mathrm{i}}
\newcommand{\e}{\mathrm{e}}
\newcommand{\dd}{\mathrm{d}}

\title{Asymptotics for a Variant of the Mittag-Leffler Function}

\author{Stefan Gerhold}

\address{Vienna University of Technology, Institute of Mathematical Methods in Economics,
Wiedner Hauptstr.\ 8/105-1,
A-1040 Vienna, Austria}

\email{sgerhold at fam.tuwien.ac.at}

\date{\today}

\begin{document}

\begin{abstract}
  We generalize the Mittag-Leffler function by attaching an exponent to its
  Taylor coefficients. The main result is
  an asymptotic formula valid in sectors of the complex plane,
  which extends work by Le Roy [Bull.\ des sciences math.~24, 1900]
  and Evgrafov [Asimptoticheskie otsenki i tselye funktsii, 1979].
  It is established by Plana's summation formula in conjunction with
  the saddle point method.
  As an application, we (re-)prove a non-holonomicity result about powers
  of the factorial sequence.
\end{abstract}



\maketitle

\section{Introduction}

For $a,b,\alpha>0$, the series
\begin{equation}\label{eq:fu}
  F_{a,b}^{(\alpha)}(z) = \sum_{n=0}^\infty \frac{z^n}{\Gamma(a n+ b)^{\alpha}}
\end{equation}
defines an entire function of~$z$. The Bessel function
$\mathrm{I}_0(2\sqrt{z})=\sum_{n=0}^\infty z^n/n!^2$ and the
(generalized) Mittag-Leffler function $E_{a,b}(z)=\sum_{n=0}^\infty z^n/\Gamma(a n+b)$
are special cases.
The function~\eqref{eq:fu} has the order $1/a\alpha$; see, e.g.,
Titchmarsh~\cite[Example~8.4]{Ti39}, for $a=b=1$. His argument trivially
extends to $a,b>0$.

Le Roy~\cite{leR00c} has obtained the asymptotics of~$F_{1,1}^{(\alpha)}(z)$
as~$z$ tends to infinity along the real line. This result
(a special case of formula~\eqref{eq:asympt} below) appears also in Hardy~\cite[p.~55]{Ha10}.
Le Roy establishes it by converting the sum into an integral and then
applying the Laplace method.
Note that the latter can also be applied in a direct way:
For real~$z$, the summands of
\begin{equation}\label{eq:fu1}
  F_{1,1}^{(\alpha)}(z) = \sum_{n=0}^\infty \frac{z^n}{n!^\alpha}
\end{equation}
are positive and concentrated
near $n\approx z^{1/\alpha}$, and the Laplace method
is easily carried out. The contribution of the present note
is an extension of the asymptotics to $a,b>0$ and complex values of~$z$,
which is presented in Section~\ref{se:main}.
(For $a=b=1$ and complex~$z$, this question is also discussed in Evgrafov's book~\cite{Ev79};
see the end of Section~\ref{se:main} for detailed comments.)

As a small application, we prove in Section~\ref{se:appl}  
that the (possibly formal) series~\eqref{eq:fu1} is not
$D$-finite~\cite{St80} for all $\alpha\in\mathbb{R}\setminus \mathbb{Q}$.
(This is a special case of a known result~\cite{BeGeKlLu08}.)

We note in passing that the function~\eqref{eq:fu1} satisfies the integral relation
\begin{equation}\label{eq:int rel}
  \int_0^\infty \e^{-t/z} F_{1,1}^{(\alpha+1)}(t) \dd t = z F_{1,1}^{(\alpha)}(z),
    \qquad z\neq 0.
\end{equation}
Indeed:
\begin{align*}
  \int_0^\infty \e^{-t/z} \sum_{n=0}^\infty \frac{t^n}{n!^{\alpha+1}} \dd t
    &= \sum_{n=0}^\infty \frac{1}{n!^{\alpha+1}}
    \int_0^\infty t^n \e^{-t/z} \dd t \\
  &= \sum_{n=0}^\infty \frac{z}{n!^{\alpha+1}}
     \int_0^\infty (zs)^n \e^{-s} \dd s
    = \sum_{n=0}^\infty \frac{z^{n+1}}{n!^\alpha}.
\end{align*}

\section{Main Result}\label{se:main}

\begin{theorem}\label{thm:main}
  Let $\alpha,a,b>0$ and $\varepsilon>0$ be arbitrary. Then, for $z\to\infty$
  in the sector
  \[
    |\arg(z)| \leq
      \begin{cases}
        \tfrac12 a\alpha \pi - \varepsilon & 0<a\alpha < 2\\
        (2-\tfrac12 a\alpha)\pi - \varepsilon & 2\leq a\alpha<4 \\
        0 & 4\leq a\alpha,
      \end{cases}
  \]
  we have the asymptotics
  \begin{equation}\label{eq:asympt}
    F_{a,b}^{(\alpha)}(z)  \sim \frac{1}{a\sqrt{\alpha}}
      (2\pi)^{\frac{1-\alpha}{2}}
      z^{\frac{\alpha-2b\alpha +1}{2 a\alpha }} \e^{\alpha z^{1/a\alpha}}.
  \end{equation}
\end{theorem}
Applying the Laplace method directly does not work for non-real~$z$;
the absolute values of the summands in~\eqref{eq:fu}
are peaked near $n\approx a^{-1}|z|^{1/a\alpha}$, but
it seems that one cannot balance the local expansion and the tails.
This is caused by oscillations in the summands, which can be dealt with
by shifting the problem to the asymptotic evaluation of integral.
The Laplace method then succeeds, after moving the integration contour
through a saddle point located approximately at~$a^{-1}z^{1/a\alpha}$.
\begin{lemma}
  Let $\alpha,a,b>0$ and $\varepsilon>0$ be arbitrary. Then, as $z\to\infty$ in the sector
  $|\arg(z)|\leq \max\{0,(2-\tfrac12 a\alpha)\pi\}$, we have
  \begin{equation}\label{eq:sum int}
    F_{a,b}^{(\alpha)}(z)
      = \int_0^\infty \frac{z^t}{\Gamma(at+b)^\alpha} \dd t + \Oh(z).
  \end{equation}
\end{lemma}
\begin{proof}
  First, let us fix a~$z$ in this sector.
  Put $f(t) = z^t/ \Gamma(at+b)^{\alpha}$. Plana's summation
  formula~\cite[Theorem~4.9~c]{He74} yields
  \begin{equation}\label{eq:plana}
    \sum_{n=0}^\infty f(n) = \int_0^\infty f(t) \dd t + \tfrac12 f(0)
      + \im \int_0^\infty \frac{f(\im t)-f(-\im t)}{\e^{2\pi t}-1}\dd t.
  \end{equation}
  To check the validity of~\eqref{eq:plana}, we have to verify that
  \[
    \lim_{y\to\infty} |f(x\pm \im y)| \e^{-2\pi y} = 0,
  \]
  uniformly for~$x$ in finite intervals in $[0,\infty[$, and that
  \[
    \int_0^\infty  |f(x\pm \im y)| \e^{-2\pi y} \dd y
  \]
  exists for $x\geq0$ and tends to zero for $x\to\infty$.
  To do so, first note that
  \[
    |z^{x \pm \im y}| = |z|^x \e^{\mp y \arg(z)} \leq |z|^x \e^{ y |\arg(z)|}.
  \]
  Furthermore, by Stirling's formula, we have
  \[
    |\Gamma(x \pm \im y)^{-\alpha}| \leq \e^{(\alpha+\delta) x} x^{-\alpha x}
    \e^{ (\alpha\pi/2 + \delta) y}
  \]
  for large $|x+\im y|$, where~$\delta>0$ is arbitrary. Hence
  \[
    |\Gamma(ax+b \pm \im ay)^{-\alpha}| \leq \e^{(a\alpha+\delta) x}
      (ax+b)^{-\alpha(a x+b)} \e^{ (a\alpha\pi/2 + \delta) y},
  \]
  and thus
  \begin{equation}\label{eq:f est}
    |f(x\pm \im y)| \leq \frac{\e^{(a\alpha+\delta) x}|z|^x}{(ax+b)^{\alpha (ax+b)}}
      \e^{(|\arg(z)|+a\alpha\pi/2 + \delta)y},
  \end{equation}
  which implies both required conditions.
  Finally, putting $x=0$, we see
  from~\eqref{eq:f est} that the second integral in~\eqref{eq:plana}
  is~$\Oh(z)$ as $z\to\infty$. Since $f(0)=\Oh(1)$, we are done.
\end{proof}

\begin{proof}[Proof of Theorem~\ref{thm:main}]
  We apply the saddle point method to the integral in~\eqref{eq:sum int}.
  To locate the saddle point, we
  equate the derivative of the logarithm of the integrand to zero,
  which leads to the equation
  \[
    0=\log z - a\alpha\frac{\Gamma'(at+b)}{\Gamma(at+b)} =
      \log z-a\alpha \log (at+b) + \frac{a\alpha}{2(at+b)} + \Oh(t^{-2}).
  \]
  By bootstrapping, we find that there is an approximate saddle point at
  \[
    t_0 := a^{-1}z^{1/a\alpha} + \tfrac{1-2b}{2a}.
  \]
  We change the integration contour to a line~$\mathcal{L}$ that begins at~$0$
  and passes through~$t_0$. Note that this change of contour is valid for large~$|z|$,
  by Stirling's
  formula, as long as $|\arg(t_0)|$ is bounded away from~$\pi/2$.
  But this follows from our assumption on~$\arg(z)$.
  
  The dominant contribution of the integral arises from the range
  \[
    |t-t_0| \leq |t_0|^{\beta}
  \]
  around the saddle point, where~$\beta$ is an arbitrary member of the
  interval~${]\tfrac12,\tfrac23[}$.
  We write
  \[
    t=t_0(y+1), \qquad -1\leq y< \infty,
  \]
  and divide the integral as follows:
  \begin{align}
    \int_\mathcal{L} \frac{z^t}{\Gamma(at+b)^\alpha} \dd t &=
      t_0 \int_{-1}^\infty  \frac{z^{t_0(y+1)}}
      {\Gamma(a t_0(y+1)+b)^\alpha} \dd y \notag \\
    &= t_0\left( \int_{-1}^{-|t_0|^{\beta-1}}+ \int_{-|t_0|^{\beta-1}}^{|t_0|^{\beta-1}}
      + \int_{|t_0|^{\beta-1}}^\infty \right)  
      \frac{z^{t_0(y+1)}}{\Gamma(a t_0(y+1)+b)^\alpha} \dd y \notag \\
    &=: I_1 + I_2 + I_3.     \label{eq:div}
  \end{align}
  First we investigate the central integral~$I_2$.
  {}From Stirling's formula
  we get the following local expansion of the logarithm of the integrand:
  \begin{multline}\label{eq:loc expans}
    (t_0(y+1)) \log z -\alpha \log \Gamma((a t_0(y+1)+b) =\\
       t_0 \log z - a\alpha t_0 \log t_0
      +a\alpha(1-\log a) t_0 + \alpha(\tfrac12 -b) \log t_0 \\
       + \alpha(\tfrac12-b)\log a  - \alpha \log \sqrt{2\pi}
      -\tfrac12 a\alpha t_0 y^2 + \oh(1).
  \end{multline}
  Since
  \begin{align*}
    \int_{-|t_0|^{\beta-1}}^{|t_0|^{\beta-1}} \exp(-\tfrac12 a\alpha t_0 y^2) \dd y
    &= \frac{1}{\sqrt{|t_0|}} \int_{-|t_0|^{\beta-1/2}}^{|t_0|^{\beta-1/2}}
      \exp(-\tfrac12 a\alpha t_0 |t_0|^{-1} u^2) \dd u  \\
    &\sim \frac{1}{\sqrt{|t_0|}} \int_{-\infty}^{\infty}
      \exp(-\tfrac12 a\alpha t_0 |t_0|^{-1} u^2) \dd u \\
    &= \sqrt{\frac{2\pi}{a\alpha t_0}},
  \end{align*}
  the central part~$I_2$ thus satisfies
  \begin{align}
    I_2 &\sim t_0 z^{t_0} t_0^{-a\alpha t_0}\e^{a\alpha(1-\log a) t_0}
    (at_0)^{\alpha(\tfrac12-b)}
      (2\pi)^{-\alpha/2} \sqrt{\frac{2\pi}{a\alpha t_0}} \label{eq:I1} \\
    &\sim \frac{1}{a\sqrt{\alpha}}
      (2\pi)^{\frac{1-\alpha}{2}}
      z^{\frac{\alpha-2b\alpha +1}{2 a\alpha }} \e^{\alpha z^{1/a\alpha}}. \notag
  \end{align}
  This is the right hand side of~\eqref{eq:asympt}.
   
  It remains to show that the integrals~$I_1$ and~$I_3$ are negligible.
  By our assumption on~$\arg(z)$, there is a
  constant $c_1>0$ (independent of~$z$) such that
  \[
    |\Im(t_0)| \leq c_1 \Re(t_0),
  \]
  hence
  \begin{equation}\label{eq:Re t_0}
    \Re(t_0) \geq \frac{|t_0|}{\sqrt{c_1^2+1}} =: c_2 |t_0|.
  \end{equation}
  Now divide the integral~$I_3$ further into
  \begin{align*}
    I_3 &= t_0\left( \int_{|t_0|^{\beta-1}}^u  + \int_u^\infty \right)  
      \frac{z^{t_0(y+1)}}{\Gamma(a t_0(y+1)+b)^\alpha} \dd y \\
      &=: I_{31} + I_{32},
  \end{align*}
  where
  \[
    u := \exp(\tfrac13 a\alpha c_2 |t_0|^{2\beta-1}).
  \]
  By Stirling's formula, there is a positive constant~$c_0$ such that
  \begin{multline}\label{eq:for tail}
    \left| \frac{z^t}{\Gamma(at+b)^\alpha} \right| \leq c_0
      \exp \Re\Bigl((at+b) \log z - \alpha (at+b)\log (at+b) \\
      +\alpha (at+b) +\tfrac12 \alpha \log (at+b) \Bigr),
      \qquad t\in\mathcal{L}.
  \end{multline}
  An elementary calculation shows that the right hand side of~\eqref{eq:for tail}
  decreases w.r.t.~$|y|$ for large~$|z|$ and $|y|\geq |t_0|^{\beta-1}$.
  Therefore, we can estimate~$I_{31}$
  by inserting $y=|t_0|^{\beta-1}$ into~\eqref{eq:for tail} and multiplying
  by the length of the integration path, which is $u-|t_0|^{\beta-1}< u$.
  Using~\eqref{eq:loc expans}, and writing~$A=A(z)$ for the factor in front of
  the square root in~\eqref{eq:I1}, we obtain
  \begin{align*}
    \left| \frac{z^t}{\Gamma(at+b)^\alpha} \right|_{y=|t_0|^{\beta-1}}
      &\leq c_0  \left|A \e^{-a\alpha t_0|t_0|^{2\beta-2}/2} \right| \\
    &= c_0|A|\e^{-a\alpha \Re(t_0)|t_0|^{2\beta-2}/2}
    \leq c_0|A|\e^{-a\alpha c_2 |t_0|^{2\beta-1}/2}.
  \end{align*}
  The latter inequality follows from~\eqref{eq:Re t_0}. 
  Hence
  \[
    |I_{31}| \leq c_0 u|A|\cdot \e^{-a\alpha c_2 |t_0|^{2\beta-1}/2}
      = c_0|A| \e^{-a\alpha c_2 |t_0|^{2\beta-1}/6}.
  \]
  Now we compare this estimate with~\eqref{eq:I1}. Since
  \[
     \e^{-a\alpha c_2 |t_0|^{2\beta-1}/6} \ll |t_0|^{-1/2},
  \]
  the integral~$I_{31}$ is indeed negligible. As for~$I_{32}$, it easily
  follows from Stirling's formula that
  \[
    \left| \frac{z^t}{\Gamma(at+b)^\alpha} \right| \leq \e^{-y}, \qquad y\geq u,
  \]
  for large~$|z|$. We thus obtain
  \begin{align*}
    |I_{32}| &\leq t_0 \int_u^\infty \e^{-y} \dd y \\
    &= t_0\e^{-u} \ll I_2.
  \end{align*}
  Finally, the integral~$I_1$ in~\eqref{eq:div} can be estimated analogously
  to~$I_{31}$.
  %
\end{proof}
A full asymptotic expansion can be obtained easily by pushing the
local expansion around the saddle point further.

Evgrafov~\cite[\S~4.2]{Ev79} offers a similar asymptotic treatment
of~$F_{1,1}^{(\alpha)}(z)$ in sectors of the complex plane. For
$\alpha<2$ and $|\arg(z)|<\tfrac12 \alpha\pi-\varepsilon$, his
result agrees with ours. (Except that a factor $\alpha^{-1/2}$,
or $\rho^{1/2}$ in Evgrafov's notation, is missing from the formula.)
However, he gives few details on how to carry out the saddle
point analysis, in particular, on how to do the tail estimates.
For $\alpha\geq2$, Evgrafov~\cite[p.~294]{Ev79} appears to go beyond our
Theorem~\ref{thm:main}, in that he claims~\eqref{eq:asympt}
(with $a=b=1$)
for \emph{any} sector that stays away from the negative real
axis. There seems to be a serious gap in the proof, though.

To be specific, we switch to Evgrafov's notation.
On p.~292, he writes that $\sum_{n=0}^\infty t^n n!^{-1/\rho}$
satisfies the assumptions of Theorem~4.2.2 for all $\rho>0$.
(The text says Theorem~3.2.2 instead, but this is certainly a typo.)
This means that $\mu(z)=\Gamma(z+1)^{-1/\rho}$ should
satisfy $|\mu(x+\im y)|< M_A \exp(-A x)$, for arbitrary~$A$ 
and some other constant~$M_A$, and for $x+\im y$ in a
domain~$D$ containing a horizontal strip that contains
the positive real line.
But due to the exponential decrease
of the Gamma function towards $\pm\im\infty$, this can
hold only if the elements of~$D$ have bounded imaginary part.
Then also the contours~$C_{-1}^+$ and~~$C_{-1}^-$ on p.~292
must have bounded imaginary parts.
On p.~293, the saddle point method is applied to the second integral
on p.~292 (over the contour~$C_{-1}^+$),
where the location of the saddle point is $t^\rho \exp(2\pi\im\rho)$.
The imaginary part of this point is \emph{not} bounded for large~$t$.
Hence the validity of the necessary change
of integration contour is not proven.
These remarks seem to justify another study of the problem, provided by the
present note.

We close the section by a possible question for future research.
For complex~$\alpha$ and fixed~$z$, one might
ask whether the function defined by~\eqref{eq:fu1} has an analytic continuation
for $\Re(\alpha)\leq 0$. (The relation~\eqref{eq:int rel} does not seem
to be useful in this respect.)

\section{An Application: Non-Holonomicity}\label{se:appl}
An analytic function (or formal power series) is called $D$-finite~\cite{St80},
or holonomic, if it satisfies a linear ODE with polynomial coefficients.
An equivalent condition is that its power series coefficients satisfy
a linear recurrence with polynomial coefficients.

There has been some interest recently in showing that certain series
(resp.\ sequences) are \emph{not}
holonomic~\cite{BeGeKlLu08,FlGeSa:05,FlGeSa10,Ge04,LuSt11,MiRe09}.
Lipshitz~\cite[Example~3.4(i)]{Li89}
mentions (without proof) that the sequence $(n!^\alpha)$, which satisfies
\[
  (n+1)!^\alpha = (n+1)^\alpha n!^{\alpha},
\]
is holonomic
if and only if $\alpha$ is an integer.
See~\cite[Theorem~4.1]{BeGeKlLu08} for a proof that
there is indeed no recurrence with polynomial coefficients for $\alpha\in
\mathbb{C}\setminus\mathbb{Z}$.

From Theorem~\ref{thm:main}, we can conclude the weaker result
that~$F_{1,1}^{(\alpha)}(z)$, and thus $(n!^\alpha)$, is not holonomic for $\alpha\in
\mathbb{R}\setminus\mathbb{Q}$. Indeed, since multiplication by the holonomic
sequence $n!^{-\lfloor \alpha \rfloor}$ preserves holonomicity, we may assume
that $0 \leq \alpha < 1$. Hence Theorem~\ref{thm:main} yields the asymptotics
of~$F_\alpha(z)$. But a function that features the element
$\exp(\alpha z^{1/\alpha})$ in its asymptotic expansion at infinity,
in a sector of positive opening angle, can be holonomic
only for rational~$\alpha$. This follows from a classical result on the asymptotic behavior
of solutions of linear ODEs.
(For details on this method of showing
non-holonomicity, see~\cite{FlGeSa:05} and~\cite{FlGeSa10}.)
Including the parameters~$a$ and~$b$, one can conclude other non-holonomicity
results from Theorem~\ref{thm:main}, but they are going to be weaker
(in terms of parameter ranges) than the corresponding results deduced from
the method of~\cite{BeGeKlLu08}.

\bigskip
{\bf Acknowledgment.} I thank Julia Eisenberg for her linguistic help concerning
the relevant passages of Evgrafov's book~\cite{Ev79}.

\bibliographystyle{siam}
\bibliography{../gerhold}

\end{document}